\documentclass[10pt,a4paper]{article}

\usepackage{amssymb,amsmath}
\usepackage{color}
\usepackage{graphicx,graphics}
\usepackage{mathtools}
\usepackage[english]{babel}
\usepackage[utf8]{inputenc}
\usepackage{epsfig,url}
\usepackage{bbm,theorem}
\usepackage{a4wide}
\usepackage{enumerate}
\usepackage{calrsfs}
\DeclareMathAlphabet{\pazocal}{OMS}{zplm}{m}{n}

\newtheorem{theorem}{Theorem}[section]

\newtheorem{lemma}[theorem]{Lemma}
\newtheorem{proposition}[theorem]{Proposition}

\def\w#1{\mathop{:}\nolimits\!#1\!\mathop{:}\nolimits}

{\theorembodyfont{\upshape}

}
\numberwithin{equation}{section}
\numberwithin{theorem}{section}

\newcommand{\qed}{\hfill$\Box$}
\newenvironment{proof}{\begin{trivlist}\item[]{\em Proof:}\/}{%
\qed\end{trivlist}}

\newcommand{\E}{{\mathbb E}}
\newcommand{\Z}{{\mathbb Z}}
\newcommand{\R}{{\mathbb R}}
\newcommand{\C}{{\mathbb C\hspace{0.05 ex}}}
\newcommand{\N}{{\mathbb N}}
\newcommand{\T}{{\mathbb T}}

\newcommand{\cf}{{\mathbbm 1}}

\newcommand{\ci}{{\rm i}}

\newcommand{\rme}{{\rm e}}
\newcommand{\rmd}{{\rm d}}

\newcommand{\FT}[1]{\widehat{#1}}

\DeclareMathOperator*{\cov}{Cov}
\DeclareMathOperator*{\var}{Var}

\newcommand{\norm}[1]{\Vert #1\Vert}

\newcommand{\set}[1]{\{#1\}}

\newcommand{\mean}[1]{\langle #1\rangle}

\newcommand{\defem}[1]{{\em #1\/}}
\newcommand{\qand}{\quad\text{and}\quad}

\newcounter{jlisti}

\newcommand{\Cgen}{g_{\mathrm{c}}}
\newcommand{\Wgen}{G_{\mathrm{w}}}
\newcommand{\Mgen}{G_{\mathrm{m}}}

\newcommand{\be}{\begin{equation}}
\newcommand{\bes}{\begin{equation*}}
\newcommand{\en}{\end{equation}}
\newcommand{\ens}{\end{equation*}}
\date{\today}

\title{Summability of joint cumulants of nonindependent lattice fields}
\author{Jani Lukkarinen\thanks{\emailjani} , Matteo Marcozzi\thanks{\emailmatteo} , Alessia Nota\thanks{\emailalessia} \\[1em]
$\,^*,\,^\dag,\,^\ddag$\UHaddress\\[0.5em]
$\,^\ddag$\UBaddress}
\date{\today}
\newcommand{\email}[1]{E-mail: \tt #1}
\newcommand{\emailjani}{\email{jani.lukkarinen@helsinki.fi}}
\newcommand{\emailmatteo}{\email{matteo.marcozzi@helsinki.fi}}
\newcommand{\emailalessia}{\email{nota@iam.uni-bonn.de}}
\newcommand{\UHaddress}{\em University of Helsinki, Department of Mathematics and Statistics\\
\em P.O. Box 68, FI-00014 Helsingin yliopisto, Finland}
\newcommand{\UBaddress}{\em University of Bonn, Institute for Applied Mathematics\\
\em Endenicher Allee 60, D-53115 Bonn, Germany}

\begin{document}
\maketitle

\begin{abstract}
We consider two nonindependent
random fields $\psi$ and $\phi$ defined on a countable set $Z$.  For instance, 
$Z=\Z^d$ or $Z=\Z^d\times I$, where $I$ denotes a finite set of possible ``internal degrees of freedom'' such as spin.
We prove that, if the cumulants of both $\psi$ and $\phi$ are $\ell_1$-clustering up to order $2 n$, 
then all joint cumulants between $\psi$ and $\phi$ are 
$\ell_2$-summable up to order $n$, in the precise sense described in the text.
We also provide explicit estimates in terms of the related $\ell_1$-clustering norms, and derive a weighted $\ell_2$-summation 
property of the joint cumulants if the fields are merely $\ell_2$-clustering.
One immediate application of the results is given by a stochastic process $\psi_t(x)$
whose state is $\ell_1$-clustering at any time $t$: then the above estimates can be applied with 
$\psi=\psi_t$ and $\phi=\psi_0$ and we obtain uniform in $t$ 
estimates for the summability of time-correlations of the field.
The above clustering assumption is obviously satisfied by any $\ell_1$-clustering stationary state of
the process, and our original motivation for the control of the summability of time-correlations
comes from a quest for a rigorous control of the Green-Kubo correlation function in such a system.
A key role in the proof is played by the properties of non-Gaussian Wick polynomials and their connection to cumulants.
\end{abstract}

\section{Introduction and physical motivation}

In many problems of physical interest, the basic dynamic variable is a random field.  In addition to proper stochastic processes, such as particles evolving according to Brownian motion, the random field could describe for instance a density of particles of a Hamiltonian system with random initial data or after time-averaging.  

One particular instance of the second kind is the \defem{Green-Kubo formula} which connects the transport coefficients, such as thermal conductivity, to integrals over \defem{equilibrium time-correlations} of the current observable of the relevant conserved quantity, for instance, of the energy current.  The equilibrium time-correlations are cumulants of current fields 
between time zero and some later time.  The current fields are generated by distributing the initial data according to some fixed equilibrium measure and then solving the evolution equations: this yields a random field, even when the time-evolution itself is deterministic.

Hence, the control of correlation functions, i.e., cumulants, of random fields is a central problem for a rigorous study of transport properties.  One approach, which has been used both in practical applications and in direct mathematical studies,
is given by Boltzmann transport equations.  It is usually derived from the microscopic system by using moment hierarchies, such as the BBGKY hierarchy, and then ignoring higher order moments to close the hierarchy of evolution equations.  
Although apparently quite powerful a method, it has not been possible to give any meaningful general estimates for the accuracy or for regions of applicability of such closure approximations.  

The present work arose as part of a project aiming at a rigorous derivation of a Boltzmann transport equation in the kinetic scaling limit of the weakly nonlinear discrete Sch\"{o}dinger equation (DNLS).  This system 
describes the evolution of 
a complex lattice field $\psi_t(\mathbf{x})$, with $\mathbf{x}\in \Z^d$ and $t\ge 0$, by 
requiring that is satisfies the Hamiltonian evolution equations
\begin{align}\label{eq:defNLSx}
\ci\partial_t \psi_t(\mathbf{x}) = \sum_{\mathbf{y} \in \Z^d} 
\alpha(\mathbf{x}-\mathbf{y})\psi_t(\mathbf{y}) + \lambda |\psi_t(\mathbf{x})|^2 
\psi_t(\mathbf{x}) \, ,
 \end{align}
where the function $\alpha$ determines the ``hopping amplitudes'' and $\lambda>0$ is a coupling constant.
A kinetic scaling limit with a suitably chosen closure assumption predicts that the Green-Kubo correlation function of the energy density of this system satisfies a linearized phonon Boltzmann equation in the limit; the explicit form of the Boltzmann collision operator and discussion about the approximations involved is given in Sections 5 and 6 in \cite{LM15}, and we refer to \cite{spohn05,L16} for more details about the linearization procedure.

The method used in the derivation of the transport properties in such weak coupling limits are, naturally, based on perturbation expansions.  Advances have recently been made in controlling the related oscillatory integrals (see for instance \cite{erdyau99,NLS09,erdyau05a}), but for nonlinear evolution equations of the present type a major obstacle has been the lack of useful \defem{a priori} bounds for the correlation functions.  For instance, Schwarz inequality estimates of moments in the ``remainder terms'' of finitely expanded moment hierarchies has been used for this purpose for time-stationary initial data in \cite{NLS09}, which was inspired by the bounds from unitarity of the time-evolution of certain linear evolution equations first employed for the random Schr\"{o}dinger equation in \cite{erdyau99} and later extended to other similar models such as the 
Anderson model \cite{chen03} and a classical harmonic lattice with random mass perturbations \cite{ls05}.
However, as argued in \cite{LM15}, using moments instead of cumulants to develop the hierarchy could lead to loss of an important
decay property which is valid for cumulants but not for moments; we shall discuss this point further in Section \ref{sec:setting}.

In the present contribution we derive a generic result which allows to bound joint correlations of two random fields
in terms of estimates involving only the decay properties of each of the fields separately.  These estimates are immediately
applicable for obtaining \defem{uniform in time a priori} bounds for time-correlation functions of time-stationary fields.
In particular, they imply that if the initial state of the field is distributed according to an equilibrium measure which is $\ell_1$-clustering, then all time-correlations are $\ell_2$-summable.  The precise assumptions are described in Section 
\ref{sec:setting} and the result in Theorem \ref{th_unresult} there.

If both fields are Gaussian and translation invariant, more direct estimates involving discrete Fourier-transform become available. We use this in Section \ref{sec:Gaussian} to give an explicit example which shows that $\ell_1$-clustering of the fields does
not always extend to their joint correlations, hence showing that the increase of the power from $\ell_1$-clustering to $\ell_2$-summability of the joint correlations in the main theorem is not superfluous.

The result is a corollary of a bound which proves summability of cumulants of any observable with finite variance
with $\ell_p$-clustering fields, for $p=1$ and $p=2$.  The $p=2$ case is more involved than the $p=1$ case, since the present
bound requires taking the sum in a weighted $\ell_2$-space.  The precise statements and all proofs are given in Section \ref{sec:ell2clustering}.  
The proof is based on representation of cumulants using Wick polynomials.  We rely on the results and notations of \cite{LM15}, and for convenience of the reader we have summarized the relevant items in Appendix \ref{appendix}.
We also present a few immediate applications of these bounds and 
discuss possible further applications in Section \ref{sec:discussion}.

\subsection*{Acknowledgements}

We thank Antti Kupiainen, Sergio Simonella, and Herbert Spohn for useful discussions on the topic.
The work has been supported by the 
Academy of Finland
via the Centre of Excellence in Analysis and Dynamics Research (project 271983) 
and from an Academy Project (project 258302).   The work and the related discussions have partially occurred in  workshops
supported
by the French Ministry of Education through the grant ANR (EDNHS) and also by the
Erwin Schr\"{o}dinger Institute (ESI), Vienna, Austria.

\section{Notations and mathematical setting}\label{sec:setting}

We consider here \defem{complex lattice fields} $\psi:Z\to \C$ where $Z$ is any nonempty \defem{countable} index set.
We focus on this particular setup since it is the one most directly relevant for physical applications:
common examples would be $Z=\Z^d$ and $Z=\Z^d\times I$, where $I$ denotes a finite set of possible ``internal degrees of freedom'' 
such as spin.  The setup can also cover more abstract index sets, such as the sequence of  
coefficients in the Karhunen--Lo\`{e}ve decomposition of a stochastic process \cite{Loeve48, Loeve78, Adler07}, 
or distribution-valued random fields evaluated at suitably chosen sequence of test-functions (details about
the definition and properties of general random fields can be found for instance in \cite{glimmjaffe,Bogachev98} and in other sources discussing
the Bochner--Minlos theorem).

We also assume that the field is \defem{closed under complex conjugation}:
to every $x\in Z$ there is some $x_*\in Z$ for which $\psi(x)^* = \psi(x_*)$.
If needed, this can always be achieved by replacing the original index set $Z$ by $Z\times\set{-1,1}$ and defining
a new field $\Psi$ by setting $\Psi(x,1)=\psi(x)$ and $\Psi(x,-1)=\psi(x)^*$.  This procedure was in fact used 
 in \cite{LM15,NLS09} to study the DNLS example mentioned above,  
resulting in the choice $Z=\Z^d\times \set{-1,1}$.

A \defem{random lattice field} on $Z$ is then a collection of random variables $\psi(x)$, $x\in Z$,
on the probability space $(\Omega,\mathcal{M},\mu)$,
where $\Omega$ denotes the sample space, $\mathcal{M}$ the $\sigma$-algebra of measurable events,
and $\mu$ the probability measure. 
We consider here two random fields $\psi$ and $\phi$ which are defined on the same probability space.
We denote the expectation over the measure $\mu$ by $\E$. 

The \defem{$n$:th connected correlation function} $u_n$ of the field $\psi$ is a
map $u_n:Z^n\to \C$ which is defined as the cumulant of the $n$ random variables
obtained by evaluating the field at the argument points; explicitly, 
\begin{align}
 u_n(x) := \kappa[\psi(x_1),\ldots,\psi(x_n)]\, , \qquad x\in Z^n.
\end{align}
We employ here the notations and basic results for cumulants and the related Wick polynomials, as derived in \cite{LM15}:
a summary of these is also included in Appendix \ref{appendix}.

In physics, one often encounters random fields defined on the $d$-dimensional cubic lattice, with $Z=\Z^d$.
One could then study the decay properties of such functions as $|x|\to \infty$ by using the standard
$\ell_p$-norms over $(\Z^d)^n$.  However, this is typically too restrictive for physical applications:
it would imply in particular that both the first and the second cumulant, i.e., the mean and the variance, of the random variable $\psi(x)$ 
decay as $|x|\to \infty$, and thus the field would be almost surely ``asymptotically zero'' at infinity.
Instead, many stationary measures arising from physical systems are \defem{spatially translation invariant}:
the expectation values remain invariant if all of the fields $\psi(x)$ are replaced by $\psi(x+x_0)$ 
for any given $x_0\in \Z^d$.  Since this implies also translation invariance of all correlation functions,
they cannot decay at infinity then, unless the field is almost surely zero \defem{everywhere}.

To cover also such nondecaying stationary states, one uses instead of the direct $\ell_p$-norms of the function $u_n$, 
the so-called \defem{$\ell_p$-clustering norms} of the field $\psi$ defined as follows: for $1\le p<\infty$ and $n\in \N_+$ we set
\begin{align}\label{eq:defclnorm}
 \norm{\psi}^{(n)}_p := \sup_{x_0\in Z} \biggl[\sum_{x\in Z^{n-1}} 
 \left|\kappa[\psi(x_0),\psi(x_1),
 \ldots,\psi(x_{n-1})]\right|^p \biggr]^{1/p}\, ,
\end{align}
and define analogously $\norm{\psi}^{(1)}_p := \sup_{x_0\in \Z^d} |\E[\psi(x_0)]|$ (note that $\kappa[X]=\E[X]$ for any random variable $X$).  We shall also use the corresponding $p=\infty$ norms, which coincide with the standard
sup-norms of $u_n$, namely, $\norm{\psi}^{(n)}_\infty=\sup_{x\in Z^n} \left|\kappa[\psi(x_1),\psi(x_2),\ldots,\psi(x_{n})]\right|= \norm{u_n}_\infty$.  Since the norms concern $L^p$-spaces over a counting measure, they are decreasing in $p$, i.e.,
$\norm{\psi}^{(n)}_p\ge \norm{\psi}^{(n)}_{p'}$ if $p\le p'$.  (This follows from the bound $|u_n(x)|\le \norm{\psi}^{(n)}_p$, valid for all $x$ and $p$.)

For a translation invariant measure with $Z=\Z^d$, we can translate $x_0$ to the origin in the definition (\ref{eq:defclnorm}),
and, by a change of variables $x_n=x_0+y_n$, obtain the simpler expression
\begin{align}
 \norm{\psi}^{(n)}_p = \biggl[\sum_{y\in (\Z^d)^{n-1}} 
 \left|\kappa[\psi(0),\psi(y_1),
 \ldots,\psi(y_{n-1})]\right|^p \biggr]^{1/p}\, .
\end{align}
The summation here goes over the \defem{displacements} $y_i$ of the argument $x_{i}$ from the reference position $x_0=0$.
The definition is tailored for random fields which become asymptotically independent
for far apart regions of the lattice, i.e., when $|y_i|\to \infty$ above.  
For translation non-invariant measures, finiteness of the norm (\ref{eq:defclnorm}) yields a uniform 
estimate for the speed of asymptotic independence of the field. Let us use the opportunity to stress that it is crucial to use the cumulants, not moments, 
above: similar moments of the field would not decay as the separation grows, even if the field values would become independent (see \cite{LM15} for more discussion
about this point).

We now call a random field $\psi$ \defem{$\ell_p$-clustering} if 
$\norm{\psi}^{(n)}_p<\infty$ for all $n=1,2,\ldots$.  In particular, this requires that all of the cumulants, which define
the connected correlation functions $u_n$, need to exist.  From the iterative definition of cumulants mentioned in the Appendix, or from the inversion formula expressing cumulants in terms of moments,
it clearly suffices that $\E[|\psi(x)|^n]<\infty$ for all $x\in Z$.
We also say that the field $\psi$ is \defem{$\ell_p$-clustering up to order $m$} if $\norm{\psi}^{(n)}_p<\infty$ for all $n\le m$.
For such a field, we use the following constants to measure its ``magnitude'': we set
\begin{align}\label{eq:defMN}
 M_N(\psi;p) := \max_{1\le n\le N} \left(\frac{1}{n!}\norm{\psi}^{(n)}_p\right)^{1/n}\, .
\end{align}
Clearly, the definition yields an increasing sequence in $N$ up to the same order in which
the field is $\ell_p$-clustering.  We use the constants $M_N$ to control the increase of the clustering norms.  It is conceivable that
in special cases other choices beside (\ref{eq:defMN}) could be used with the estimates below to arrive at sharper bounds than those stated in the theorems. 
However, the above choice is convenient for our purposes since it leads to simple combinatorial estimates, increasing typically only factorially in the degree of the cumulant.
It is possible to think of the numbers $M_N$ as measuring the range of values the field can attain.  For instance, 
if $Z=\set{0}$ and $\psi(0)$ is a random variable which almost surely belongs to the interval $[-R,R]$ with $R>0$, then $M_n(\psi;p)$
is independent of $p$ (since there is only one point $0$) and $M_n=c_n R$ where $c_n$ remains order one, uniformly in $n$ (see, for instance, Lemma C.1 in \cite{ls05}).

After these preliminaries, we are ready to state the main result:
\begin{theorem}\label{th_unresult}
 Suppose $\psi$ and $\phi$ are random lattice fields which are closed under complex conjugation and defined on the same probability space.  Assume that $\phi$ is $\ell_1$-clustering and $\psi$ is $\ell_\infty$-clustering, both up to order $2 N$ for some $N\in \N_+$.
 Then their joint cumulants satisfy the following $\ell_2$-estimate for any $n,m\in \N_+$ for which $n,m\le N$,
 \begin{align}
  & \sup_{x'\in Z^{m}} 
  \biggl[\sum_{x \in Z^n} \big\vert \kappa[\psi(x'_1),
 \ldots,\psi(x'_m),\phi(x_1),\ldots,\phi(x_n)] \big\vert^2 \biggr]^{1/2}
% \nonumber \\  & \quad 
 \leq (\mathfrak{M}_{m,n} \gamma^m)^{n+m} (n+m)! \, ,
\end{align}
where $\mathfrak{M}_{m,n} := \max(M_{2m}(\psi; \infty),M_{2n}(\phi;1))$ and $\gamma= 2 \rme\approx 5.44$.
In particular, all of the above sums are then finite.
\end{theorem}
Loosely speaking, one can say that an $\ell_1$-clustering random field can have at worst $\ell_2$-summable joint correlations.
We have stated the result in a form which assumes that the field $\psi$ is $\ell_\infty$-clustering.  
As mentioned above, the clustering norms are decreasing in the index: hence, the above result also holds if
$\psi$ is $\ell_q$-clustering for any $1\le q<\infty$.  One could then also replace the constants 
$\mathfrak{M}_{m,n}$ using the corresponding $\ell_q$-clustering norms, $\max(M_{2m}(\psi; q),M_{2n}(\phi;1))$.  However, 
these constants are always larger than $\mathfrak{M}_{m,n}$ and thus can only worsen the bound.

This result is a consequence of a more general covariance bound given in Theorem \ref{th_hom}.  
There we also give a version of the estimate for fields $\phi$ which are merely $\ell_2$-clustering.
The price to pay for the relaxation of the norms is an appearance of a weight factor in the $\ell_2$-summation,
see Theorem \ref{th_weigheted} for the precise statement.
Before going into the details of the proofs, let us
go through a special case clarifying the assumptions and the result.

\section{An example: translation invariant Gaussian lattice fields}
\label{sec:Gaussian}

In this section, we consider real valued
Gaussian random fields $\psi$ and $\phi$ on $Z=\Z$ and assume that both fields have a zero mean and are invariant under spatial translations.  Their joint measure is then determined by giving three functions $F_1$, $F_2$, $G\in \ell_2(\Z,\R)$ for which
\begin{align}\label{eq:Gaussex}
 \mean{\psi(x)\psi(y)} = F_1(x-y)\, , \quad
 \mean{\phi(x)\phi(y)} = F_2(x-y)\, , \quad
 \mean{\psi(x)\phi(y)} = G(x-y)\,.
\end{align}
The covariance operator needs to be positive semi-definite.  By first using Parseval's theorem
and then computing the eigenvalues of the remaining $2\times 2$ -matrix, we find that this is guaranteed by requiring that the Fourier-transforms of the above functions, all of which belong to 
$L^2(\T)$, satisfy almost everywhere
\begin{align}
 \FT{F}_1(k)\ge 0\, , \quad 
 \FT{F}_2(k)\ge 0\, , \quad 
 |\FT{G}(k)|^2\le \FT{F}_1(k)\FT{F}_2(k)\, .
\end{align}
These three conditions hence suffice for the existence of a unique Gaussian measure on distributions on $\Z$ 
%$\ell^2(\Z,\R)$ 
satisfying (\ref{eq:Gaussex}); details about such constructions are given for instance in \cite{Bogachev98, Abrahamsen97}.

The last condition restricts the magnitude of the correlations, and it implies that if each of the above fields is
$\ell_2$-clustering, then their correlations are $\ell_2$-summable (simply because then $\FT{G}(k) \in L^2(\T)$, and thus its inverse Fourier transform
gives a function $G\in \ell_2(\Z)$).  Hence, one might wonder if the main theorem could, in fact, be strengthened to show that $\ell_1$-clustering of the fields 
implies $\ell_1$-summability of the joint correlations.  The following example shows that this is not the case.

\subsection{$\ell_1$-clustering fields whose joint correlations are not $\ell_1$-summable}

Let us consider two i.i.d.\ Gaussian fields $\psi$ and $\phi$ whose correlations are determined by the function
\begin{align}
 G(x) = \frac{1}{\pi x} \sin\left(\frac{\pi }{2 }x\right) , \ x\ne 0\,,\qquad G(0)=\frac{1}{2}\, .
\end{align}
For such i.i.d.\ fields $F_1(x)=\cf(x=0)=F_2(x)$ which is equivalent to $\FT{F}_1(k)=1=\FT{F}_2(k)$ for all $k\in \T$.  Now for all $x\in \Z$, clearly
\begin{align}
 G(x) = \int_{-1/4}^{1/4}\!\rmd k\, \rme^{\ci 2\pi x k} \, ,
\end{align}
and thus $\FT{G}(k) = \cf(|k|<\frac{1}{4})\le 1=\sqrt{\FT{F}_1(k)\FT{F}_2(k)}$.  Therefore, such $G$ indeed defines a possible correlation between the fields $\psi$ and $\phi$.

For such Gaussian fields, all cumulants of order different from $n=2$ are zero.  We also have $\sup_{x\in \Z} \sum_{y\in \Z} |F_1(x-y)|=1$,
and, as $F_2=F_1$,
both fields are $\ell_1$-clustering, with $\norm{\psi}^{(2)}_1=1=\norm{\phi}^{(2)}_1$ and
$\norm{\psi}^{(n)}_1=0=\norm{\phi}^{(n)}_1$ for any other $n$.  However, their joint correlations satisfy for any $x'\in \Z$
\begin{align} 
 \sum_{x\in \Z} |\kappa[\psi(x'),\phi(x)]| = \sum_{y\in \Z} |G(y)| =\frac{1}{2} + 2 \sum_{y=1}^{\infty} 
 \frac{1}{\pi y} \left|\sin\left(\frac{\pi }{2 }y\right)\right|=\frac{1}{2} + \frac{2}{\pi} \sum_{n=0}^{\infty} 
 \frac{1}{2 n+1} = \infty\, .
\end{align}
Thus the joint correlations are not $\ell_1$-summable.

In contrast, $\sup_{x'}\sum_{x} |\kappa[\psi(x'),\phi(x)]|^2<\infty$, since it is equal to $\sum_y |G(y)|^2$ and $G\in \ell^2(\Z)$.

\section{$\ell_2$-summability of joint correlations of $\ell_p$-clustering fields}
\label{sec:ell2clustering}

\begin{theorem}\label{th_hom}
 Consider a random lattice field $\phi$ on a countable set $Z$, defined on a probability space 
 $(\Omega,\mathcal{M},\mu)$ and closed under complex conjugation.  Suppose that $\phi$ is $\ell_p$-clustering
 up to order $2 N$ for some $N\in \N_+$, and let $M_N(\phi;p)$ be defined as in (\ref{eq:defMN}).
 Suppose also $X\in L^2(\mu)$, i.e., $X$
  is a random variable with finite variance.  
\begin{enumerate}
 \item If $p=1$ and $n\le N$, we have a bound
 \begin{align}
  & 
  \biggl[\sum_{x \in Z^n} \big\vert \kappa[X,\phi(x_1),\ldots,\phi(x_n)] \big\vert^2 \biggr]^{1/2}
% \nonumber \\  & \quad 
 \leq \sqrt{\cov(X^*,X)} M_{2 n}(\phi;1)^n \rme^n \sqrt{(2 n)!} \, .
\end{align}
 \item If $p=2$ and $n\le N$, we have a bound
 \begin{align}
  & \sup_{x'\in Z^n} 
  \biggl[\sum_{x \in Z^n} |\Phi_n(x',x)| \big\vert \kappa[X,\phi(x_1),\ldots,\phi(x_n)] \big\vert^2 \biggr]^{1/2}
% \nonumber \\  & \quad 
 \leq \sqrt{\cov(X^*,X)} M_{2 n}(\phi;2)^{2 n} \rme^{2 n} (2 n)! \, ,
\end{align}
where $\Phi_n(x',x):=\E\!\left[\w{\phi(x'_1)^*\phi(x'_2)^*\cdots\phi(x'_n)^*}
\w{\phi(x_1)\phi(x_2)\cdots \phi(x_n)}\right]$.
\end{enumerate}
\end{theorem}

The key argument in the proof uses Wick polynomial representation of the above cumulants.  Namely, a direct consequence
of the truncated moments-to-cumulants formula given in Proposition \ref{th:W_mult_prop} in the Appendix, is that 
\begin{align}\label{eq:mainWick}
 \kappa[X,\phi(x_1),\ldots,\phi(x_n)] =  \E[X\w{\phi(x_1)\phi(x_2)\cdots\phi(x_n)}] = \E[\w{X}\w{\phi(x_1)\phi(x_2)\cdots\phi(x_n)}]\, .
\end{align}
The Proposition can be applied here since now $\E[|X|\prod_{i=1}^n |\phi(x_i)|]<\infty$ by the Schwarz inequality
estimate $\E[|X|\prod_{i=1}^n |\phi(x_i)|]^2\le \E[|X|^2]\E[\prod_{i=1}^n |\phi(x_i)|^2]$ where the first factor is finite since $X\in L^2(\mu)$, and the second factor is finite since $\phi$ is assumed to be $\ell_p$-clustering up to order $2 n$.  

Applying Schwarz inequality in (\ref{eq:mainWick}) yields a bound
\begin{align}
 \left|\kappa[X,\phi(x_1),\ldots,\phi(x_n)]\right|^2 \le \E[|{:}X{:}|^2] \E[|{:}\phi(x_1)\phi(x_2)\cdots\phi(x_n){:}|^2]
  = \cov(X^*,X) \Phi_n(x,x) \, .
\end{align}
Hence, the theorem is obviously true if $\Phi_n(x,x)$ decreases sufficiently rapidly with ``increasing'' $x$.  However, 
this is typically too restrictive: since $\Phi_1(x,x) = \E[|{:}\phi(x){:}|^2] = \var(\phi(x))$, this would require that the field $\phi$ becomes asymptotically deterministic.  The proof below combines suitably chosen test functions with the above Schwarz estimate and results in bounds which only require
summability of $\Phi_n(x',x)$ in $x$ for a fixed $x'$.  Such summability is guaranteed by the $\ell_p$-clustering of the field, and the rest of the proof consists of controlling the combinatorial factors which relate these two concepts together, cf.\ Lemma \ref{th:ell1mainest}.

Let us stress that the above result is typically \defem{not} true if \defem{moments} are used there 
instead of cumulants.  The above Schwarz inequality estimates would be straightforward for moments; in fact, such a Schwarz estimate 
was a key method in \cite{NLS09} to separate time-evolved fields from their time-zero counterparts in products of these fields.  However, the functions 
resulting from such Schwarz estimates are of the type $\E[\prod_{i=1}^n |\phi(x_i)|^2]$ and for these to be summable in $x$ the field not only has to become asymptotically deterministic, but it has to even vanish.  Cumulants of $\ell_p$-clustering fields would, on the other hand, be summable, 
but there is no obvious way of generalizing the Schwarz inequality bounds for cumulants.  The missing ingredient is here 
provided by the Wick polynomial representation (\ref{eq:mainWick}).

\begin{proof}
There is a natural Hilbert space structure associated with correlations of the present type.  We begin with test-functions $f:Z^{n}\to \C$
which have a \defem{finite} support, and define for them a (semi-)norm by the formula
\begin{align}
 & \norm{f}_{\phi,n}^2 := \E\!\left[
 \left|\smash[b]{\sum_{x\in Z^n}} f(x) \w{\phi(x)^{J_n}}\right|^2\right] 
 \quad = \sum_{x',x\in Z^n} f(x')^* f(x) \Phi_n(x',x)\, ,\nonumber \\
 & \Phi_n(x',x) := \E\!\left[\w{\phi^*(x')^{J'_n}} \w{\phi(x)^{J_n}}\right] \, ,
\end{align}
where $J_n=\set{1,2,\ldots,n}=J'_n$, and thus we have $\phi(x)^{J_n}:=\phi(x_1)\phi(x_2)\cdots \phi(x_n)$,
$\phi^*(x')^{J'_n}:=\phi(x'_1)^*\phi(x'_2)^*\cdots \phi(x'_n)^*$.
For the definition, we do not yet need any summability properties of the field $\phi$, 
it suffices that all the expectations in $\Phi_n(x',x)$ are well-defined for all $x',x$.  By the truncated moment-to-cumulants expansion of 
Wick polynomials, as given in Proposition \ref{th:W_mult_prop} in the Appendix, we have here
\begin{align}\label{eq:Phitokappas}
 \Phi_n(x',x) = \sum_{\pi \in \mathcal{P}(J'_n+ J_n)} \prod_{S \in \pi} \left(
  \kappa[\phi^*(x')_{A'},\phi(x)_{A}]
   \cf(A'\neq \emptyset, A\neq \emptyset)\right)_{A'=S|J'_n, A=S|J_n}\, ,
\end{align}
where the notation $S|J_n$ refers to the subsequence composed out of the indices belonging to $J_n$
in the cluster $S$ of the partition $\pi$ of $J'_n+J_n$.\footnote{If one has distinct labels in $J'_n$ and $J_n$, achievable always by relabelling of one of the sets, one can safely take here $J'_n+ J_n=J'_n\cup J_n$, $\mathcal{P}(J'_n+ J_n)$
equal to the ordinary partitions of the set $J'_n\cup J_n$, and also $S|J_n = S\cap J_N$.  However, such relabellings lead to unnecessarily clumsy notations in the present case, and we have opted to use the above notations from \cite{LM15}.}
The additional restrictions $A',A\neq \emptyset$ 
in the product arise from the fact that if either of them is violated, then the corresponding cluster $S$ is contained
entirely in either $J'_n$ or $J_n$, and vice versa. 
The partitions containing such a cluster are precisely those which are missing from the moments to cumulants formula by the Wick polynomial construction.
Therefore, $\Phi_n$ is finite, as soon as all cumulants up to order $2 n$ are finite.  On the other hand, this 
is already guaranteed by the assumed $\ell_p$-clustering of the field $\phi$.
For notational simplicity, let us drop the name of the field $\phi$ from the norm $\norm{f}_{\phi,n}$.

The norm can be associated with a scalar product using the polarization identity, and  we can then use it to define a Hilbert space
$\mathcal{H}_n$ by completion and dividing out the functions with zero norm, if the above formula gives only a semi-norm.
The elements of $\mathcal{H}_n$ are thus functions $f:Z^{n}\to \C$ with $\norm{f}<\infty$
(or their equivalence classes  in the semi-norm case when every $f$ and $g$ with $\norm{f-g}=0$ 
needs to be identified).  However, since we do not use these Hilbert spaces directly, let us skip the details of the construction.

We begin with joint correlations of the type $G(x):=\E[Y\w{\phi(x)^{J_n}}]$ where $Y\in L^2(\mu)$ is a random variable.
Here $G(x)$ is well defined due to the Schwarz inequality estimate $\E[|Y||{:}\phi(x)^{J_n}{:}|]^2\le \E\!\left[|Y|^2\right] \Phi_n(x,x)$.
If  $f:Z^{n}\to \C$ has a finite support, we define
\begin{align}
\Lambda[f] := \sum_{x\in Z^n} G(x) f(x) =
\E\!\left[Y \sum_{x\in Z^n} \w{\phi(x)^{J_n}} f(x)\right]  \, .
\end{align}
Applying the Schwarz inequality as above yields an upper bound
\begin{align}\label{eq:Schfest}
|\Lambda[f]|^2\le \E\!\left[|Y|^2\right]\E\!\left[
 \left|\smash[b]{\sum_{x\in Z^n}} f(x) \w{\phi(x)^{J_n}}\right|^2\right]
 = \E\!\left[|Y|^2\right] \norm{f}_n^2\, .
\end{align}

Here, by the definition of the norm, we obtain an unweighted $\ell_2$-estimate by using H\"{o}lder's inequality as follows
\begin{align}
&\norm{f}_{n}^2\le 
\sum_{x',x\in Z^n} |f(x')| |f(x)| |\Phi_n(x',x)| 
%\nonumber \\ &\quad
\le \sqrt{\sum_{x',x\in Z^n} |f(x')|^2  |\Phi_n(x',x)|}
\sqrt{\sum_{x',x\in Z^n} |f(x)|^2  |\Phi_n(x',x)|}\nonumber \\
& \quad \le \sum_{x\in Z^n} |f(x)|^2 
\sup_{x'\in Z^n} \sum_{x\in Z^n} |\Phi_n(x',x)|
\, ,
\end{align}
where we have used the obvious symmetry property $\Phi_n(x',x)^*=\Phi_n(x,x')$.
As shown below, in Lemma \ref{th:ell1mainest}, $\ell_1$-clustering of the field $\phi$ in fact
implies that there is $c_n<\infty$ such that $\sup_{x'\in Z^n}\sum_{x\in Z^n} |\Phi_n(x',x)|\le c_n$ (the explicit dependence of $c_n$ on the clustering norms is given in the Lemma).  
Hence, we can conclude that
$|\Lambda[f]|\le \sqrt{c_n \E\!\left[|Y|^2\right]}\, \norm{f}_{\ell_2}$.  Thus, thanks to the Riesz representation theorem, $\Lambda$ can be extended into a unique functional belonging to the dual of the Hilbert space $\ell_2(Z^n)$, and hence there
is a vector $\Psi\in\ell_2(Z^n)$ such that $\Lambda[f] = \sum_{x \in Z^n} \Psi(x)^* f(x)$ and $\norm{\Psi}_{\ell_2} \le \sqrt{c_n\E\!\left[|Y|^2\right]} $.
Then necessarily $G(x)=\Psi(x)^*$ for all $x$, and thus $G\in\ell_2(Z^n)$ as well, with a bound
\begin{align}\label{eq:firstbound}
\sqrt{\sum_{x\in Z^n} |G(x)|^2} \le \sqrt{c_n \E\!\left[|Y|^2\right]}\, .
\end{align}
If $Y=\w{X}$, we have $G(x)=\E[\w{X}\w{\phi(x)^{J_n}}]= \kappa[X,\phi(x_1),\ldots,\phi(x_n)]$ as explained in (\ref{eq:mainWick}), and
also $\E\!\left[|Y|^2\right]=\E\!\left[\w{X^*}\w{X}\right]=\kappa[X^*,X]=\cov(X^*,X)$.  Hence, (\ref{eq:firstbound}) implies
the bound stated in the first item.

For the weighted result, we apply (\ref{eq:Schfest}) for specially constructed test functions $f$.  
Let $F$ be any finite subset of $Z^n$ and choose an arbitrary point $y\in Z^n$.
Then $f(x)=\cf(x{\,\in\,}F) |\Phi_n(y,x)| G(x)^*$ has finite support and 
\begin{align}
 \Lambda[f] =  \sum_{x\in F} |G(x)|^2 |\Phi_n(y,x)| \le \sqrt{\E\!\left[|Y|^2\right]} \norm{f}_n <\infty\, .
\end{align}
On the other hand, we obtain the following estimate for $\norm{f}_n$
\begin{align}
&\norm{f}_{n}^2=  
\sum_{x',x\in F}  G(x)^* G(x')|\Phi_n(y,x')| |\Phi_n(y,x)| \Phi_n(x',x) \nonumber \\
&\quad \le \sqrt{
\sum_{x',x\in F} |G(x)|^2 |\Phi_n(y,x)| |\Phi_n(y,x')| |\Phi_n(x',x)|}
%\nonumber \\ &\qquad \times
\sqrt{\sum_{x',x\in F} |G(x')|^2 |\Phi_n(y,x')| |\Phi_n(y,x)| |\Phi_n(x',x)|}\nonumber \\
& \quad\le \sum_{x\in F} |G(x)|^2 |\Phi_n(y,x)|
\left(\sup_{x'\in Z^n}
\sqrt{\sum_{x\in F} |\Phi_n(x',x)|^2}\right)^2
\, ,
\end{align}
where we have used $\Phi_n(x',x)^*=\Phi_n(x,x')$ and the Schwarz inequality in the last estimate.

As shown below, in Lemma \ref{th:ell1mainest}, $\ell_2$-clustering of the field $\phi$ 
implies $\sqrt{\sum_{x\in Z^n} |\Phi_n(x',x)|^2}\le c'_n <\infty$
where the explicit dependence of $c'_n$ on the clustering norms is given in the Lemma.  Therefore,
$\Lambda[f] \le c'_n \sqrt{\E\!\left[|Y|^2\right]} \sqrt{\Lambda[f]}$.  Since $0\le \Lambda[f]<\infty$ for any subset $F$,
we can conclude that the estimate $\sqrt{\Lambda[f]} \le c'_n \sqrt{\E\!\left[|Y|^2\right]}$ also holds.  
Thus by using subsets $F=F_R$, which are constructed by choosing the first $R$ elements from a fixed enumeration of $Z^n$,
and then taking $R\to \infty$, we obtain that
\begin{align}
\sqrt{\sum_{x\in Z^n} |G(x)|^2|\Phi_n(y,x)|} \le c'_n \sqrt{\E\!\left[|Y|^2\right]}< \infty\, ,
\end{align}
for all $y\in Z^n$.  This implies the statement in the second item.
 \end{proof}

\begin{lemma}\label{th:ell1mainest}
Suppose that the field $\phi$ is closed under complex conjugation and $\ell_p$-clustering up to order $2 n$, for some $p\in [1,\infty]$ and $n\ge 1$.  
Then, for any $x'\in Z^n$,
\begin{align}\label{eq:Phinnormest}
%\sum_{x\in Z^n} |\Phi_n(x',x)|
\norm{\Phi_n(x',\cdot)}_{\ell_p}
\le \sum_{\pi \in \mathcal{P}(J_{2 n})} \prod_{S \in \pi}   \norm{\phi}^{(|S|)}_p 
\le
M_{2n}(\phi;p)^{2 n}\rme^{2 n} (2 n)! \, ,
\end{align}
where $J_{2 n} = \set{1,2,\ldots,2 n}$ and 
$\mathcal{P}(J_{2 n})$ denotes the collection of its partitions.
\end{lemma}
\begin{proof}
Let us consider some fixed $x'\in Z^n$.
We apply the Minkowski inequality to (\ref{eq:Phitokappas}),
as a function of $x$, and conclude that 
\begin{align} 
\norm{\Phi_n(x',\cdot)}_{\ell_p}
\leq 
\sum_{\pi \in \mathcal{P}(J'_n+ J_n)} \norm{F(x',\cdot\,;\pi)}_{\ell_p}
\end{align}
where 
\begin{align}\label{eq:sumPhin}
F(x',x;\pi) :=
\prod_{S \in \pi} 
\left(
  |\kappa[\phi^*(x')_{A'},\phi(x)_{A}]|
\cf(A'\neq \emptyset, A\neq \emptyset)\right)_{A'=S|J'_n, A=S|J_n}
\, .
\end{align}

Let us first consider the case $p<\infty$.  For any $\pi \in \mathcal{P}(J'_n+ J_n)$ yielding a nonzero 
$F$,
the restrictions of its clusters with $J_n$, $A=S|J_n$ in the above formula, form a partition of $J_n$.
Let us denote this partition by $\pi_2$.
Hence, we can use this partition to reorder the summation over $x \in Z^n$ into iterative summation over $x_A \in Z^A$ for $A\in \pi_2$.
Applied to (\ref{eq:sumPhin}) this yields
\begin{align}
\sum_{x \in Z^n} |F(x',x;\pi)|^p
=
\prod_{S \in \pi}
  \biggl(
  \cf(A'\neq \emptyset, A\neq \emptyset)
  \sum_{x_A \in Z^A}   
  \left|\kappa[\phi^*(x')_{A'},\phi(x)_{A}]\right|^p\biggr)_{\! A'=S|J'_n,\, A=S|J_n} \, .
\end{align}
Since the field $\phi$ is closed under complex conjugation, for each $S\in \pi$ the 
sum over $x_A$ is equal to  $\sum_{x \in Z^{A}}\left|\kappa[\phi(y')_{J_{|A'|}},\phi(x)_A]\right|^p$
where $y'=((x'_i)_*)_{i\in A'}$.
As $A'\ne \emptyset$, we may choose an element $j\in A'$.  We then denote $x_0=(x'_j)_*$.
and estimate the sum with an $ \ell_p $-clustering norm as follows
\begin{align} 
 & \sum_{x \in Z^{A}}\left|\kappa[\phi^*(x')_{A'},\phi(x)_{A}]\right|^p \le 
  \sum_{y \in Z^{|A'|-1}} \sum_{x \in Z^{A}}\left|\kappa[\phi(x_0),\phi(y)_{J_{|A'|-1}},\phi(x)_A]\right|^p \le 
  \left(\norm{\phi}^{(|A|+|A'|)}_p \right)^p\, .
\end{align}
Since $|A'|+|A|=|S|$, we can conclude that, if $p<\infty$,
\begin{align} 
\norm{F(x',\cdot\,;\pi)}_{\ell_p} \le \prod_{S \in \pi}
\norm{\phi}^{(|S|)}_p\, .
\end{align}
The corresponding estimate for $p=\infty$ is a straightforward consequence of $|\kappa[\phi^*(x')_{A'},\phi(x)_{A}]|\le  \norm{\phi}^{(|A'|+|A|)}_\infty$ which was discussed in Section \ref{sec:setting} after Eq.\ (\ref{eq:defclnorm}).

Therefore, we can now conclude that the first inequality in (\ref{eq:Phinnormest}) holds.
By the definition in (\ref{eq:defMN}), we can then apply an upper bound 
\begin{align*}
& \norm{\phi}^{(m)}_p \le m! M_m(\phi;p)^m \le m! M_{2 n}(\phi;p)^m
\, ,
\end{align*}
for any $m \le 2 n$.  If $\pi \in \mathcal{P}(J'_n+ J_n)$, we have $|S|\le 2 n$ for any $S\in \pi$, and thus
\begin{align} 
\prod_{S \in \pi} \norm{\phi}^{(|S|)}_p \le M_{2 n}(\phi;p)^{\sum_{S \in \pi}|S|} \prod_{S \in \pi} |S|! 
= M_{2 n}(\phi;p)^{2 n} \prod_{S \in \pi} |S|! 
\, .
\end{align}
A combinatorial estimate shows that 
\begin{align}\label{eq:comb_est}
\sum_{\pi \in \mathcal{P}(J_{2 n})} \prod_{S \in \pi} |S|! \le (2n)! \rme^{2 n}
\end{align}
(a proof of the inequality is available for instance in the proof of Lemma 7.3 in \cite{NLS09}).
Therefore, we have proven also the second inequality in (\ref{eq:Phinnormest}), concluding the proof of the Lemma.
\end{proof}

The following theorem contains the already stated Theorem \ref{th_unresult} in the item 1.  The remarks after the Theorem at the end of Section \ref{sec:setting} hold also in this case.
In particular, it is obviously valid for 
any $\ell_q$-clustering field $\psi$, as long as $1\le q\le \infty$.

\begin{theorem}\label{th_weigheted}
 Consider two random lattice fields $\phi(x)$ and $\psi(x)$, $x\in Z$ for a countable $Z$, defined on the same probability space 
 $(\Omega,\mathcal{M},\mu)$ and each closed under complex conjugation.  Suppose that $\phi$ is $\ell_p$-clustering and $\psi$ is $\ell_\infty$-clustering
 up to order $2 N$ for some $N\in \N_+$.  Let $M_N$ be defined as in (\ref{eq:defMN}).
 Then their joint cumulants satisfy the following $\ell_2$-estimates for any $n,m\in \N_+$ for which $n,m\le N$:
\begin{enumerate}
 \item If $p=1$, we have a bound
 \begin{align}
  & \sup_{x'\in Z^{m}} 
  \biggl[\sum_{x \in Z^n} \big\vert \kappa[\psi(x'_1),
 \ldots,\psi(x'_m),\phi(x_1),\ldots,\phi(x_n)] \big\vert^2 \biggr]^{1/2}
% \nonumber \\  & \quad 
 \leq (\mathfrak{M}_{m,n} \gamma^m)^{n+m} (n+m)! \, 
\end{align}
where $\mathfrak{M}_{m,n} :=  \max(M_{2m}(\psi; \infty), M_{2n}(\phi; 1)) $ and $ \gamma=2 \rme$.
 \item If $p=2$, we have a bound
 \begin{align}
  & \sup_{x'\in Z^{m},y\in Z^{n}} 
  \biggl[\sum_{x \in Z^n} |\Phi_n(y,x)| \big\vert \kappa[\psi(x'_1),
 \ldots,\psi(x'_m),\phi(x_1),\ldots,\phi(x_n)] \big\vert^2 \biggr]^{1/2}
 \nonumber \\  & \qquad 
 \leq (\mathfrak{M}_{m,n} \gamma^{m})^{2(n+m)} ((n+m)!)^2
\end{align}
where $\Phi_n(y,x):=\E\!\left[\w{\phi(y_1)^*\phi(y_2)^*\cdots\phi(y_n)^*}
\w{\phi(x_1)\phi(x_2)\cdots \phi(x_n)}\right]$, and we set
$\mathfrak{M}_{m,n} :=  \max(M_{2m}(\psi; \infty), M_{2n}(\phi; 2)) $ and $ \gamma=2 \rme$.
\end{enumerate}
\end{theorem}
\begin{proof}
We will proceed by induction over $m$.  Let us recall the above definition of $\Phi_n$ and define analogously
$\Psi_m(y',x'):=\E\!\left[\w{\psi(y'_1)^*\psi(y'_2)^*\cdots\psi(y'_m)^*}
\w{\psi(x'_1)\psi(x'_2)\cdots \psi(x'_m)}\right]$.  In particular, then we can apply Theorem \ref{th_hom}
with $X=\psi(x'_1)$.  By Lemma \ref{th:ell1mainest}, then $\E[|{:}X{:}|^2]=\Psi_1(x_1',x_1')\le M_{2}(\psi;\infty)^{2}\rme^2 \, 2!$,
and thus, say for $ \gamma=2 \rme$, 
both items 1 and 2 can be seen to hold for $m=1$ and any $n\le N$ thanks to Theorem \ref{th_hom} and the estimate
$(2n)!\le ((2n)!!)^2= 2^{2n}(n!)^2$.

As an induction hypothesis, we consider some $1<m\le N$ and assume that the thesis holds for values up to $m-1$ with any $n\le N$.
We also give the details only for the first $\ell_1$-clustering case, i.e., with $p=1$.  

Let us decompose the cumulant using Proposition \ref{th:W_mult_prop}.  Namely, consider 
\begin{align*}
\pazocal{P}(x', x) &:= \E[\w{\psi(x_1') \cdots \psi(x'_m)} \,\w{\phi(x_1) \cdots \phi(x_n)}] \,,
\end{align*}
for which
\begin{align*}
\kappa[\psi(x_1'), \ldots, \psi(x'_m), \phi(x_1), \ldots, \phi(x_n)] = 
\pazocal{P}(x', x)-\pazocal{Q}(x', x)\, ,
\end{align*}
with
\begin{align*}
\pazocal{Q}(x', x) &:= \sum_{\pi \in \mathcal{P}(J'_m+ J_n)} \cf(|\pi| > 1) \prod_{S \in \pi} \left(
  \kappa[\psi(x')_{A'},\phi(x)_{A}]
   \cf(A'\neq \emptyset, A\neq \emptyset)\right)_{A'=S|J'_m, A=S|J_n}\, .
\end{align*}
Then we can conclude from the Minkowski inequality that 
\begin{align}\label{K_decomp}
  \biggl[\sum_{x \in Z^n} \big\vert \kappa[\psi(x'_1),
 \ldots,\psi(x'_m),\phi(x_1),\ldots,\phi(x_n)] \big\vert^2 \biggr]^{1/2}\le
\norm{\pazocal{P}(x', \cdot)}_{\ell_2}+ \norm{\pazocal{Q}(x', \cdot)}_{\ell_2}\, .
\end{align}

We now estimate $\norm{\pazocal{P}(x', \cdot)}_{\ell_2}$ using item 1 in  Theorem \ref{th_hom}
with $X=\w{\psi(x_1') \cdots \psi(x'_m)}$.  Clearly, then
\begin{align}
\pazocal{P}(x, x') = \E[\w{X} \w{\phi(x_1) \cdots \phi(x_n)} ] = \kappa[X,\phi(x_1), \ldots,\phi(x_n) ]\,,
\end{align}
and, by applying Theorem \ref{th_hom} and Lemma \ref{th:ell1mainest}
\begin{align}\label{eq:P_bound}
\norm{\pazocal{P}(x', \cdot)}_{\ell_2} \le \sqrt{\Psi_m(x',x')}  M_{2 n}(\phi;1)^n \rme^n \sqrt{(2 n)!} 
\le
M_{2 m}(\psi; \infty)^{m}  M_{2 n}(\phi;1)^{n} \rme^{n+m} \sqrt{(2m)! (2 n)!} \, .
\end{align}
Note that $ (2 n)! \leq 2^{2n} (n!)^2$ and $n!m! \leq (n+m)! $, so, recalling the definition of $ \mathfrak{M}_{m,n}$, we have
\begin{align}\label{eq:P_bound2}
\norm{\pazocal{P}(x', \cdot)}_{\ell_2} \le (\mathfrak{M}_{m,n} 2 \rme)^{n+m} (n+m)! \, .
\end{align}

To control the second term in \eqref{K_decomp}, we first use the Minkowski inequality to the sum over the partitions, as in the proof of Theorem  \ref{th_hom}, yielding 
\begin{align*}
& \norm{\pazocal{Q}(x',\cdot)}_{\ell_2}
\\ \nonumber & 
\leq \sum_{\pi \in \mathcal{P}(J'_m+ J_n)} \!\cf(|\pi| > 1) \prod_{S \in \pi} \left(
 \biggl[\sum_{x\in Z^A} | \kappa[\psi(x')_{A'},\phi(x)_{A}]|^2\biggr]^{1/2}
   \cf(A'\neq \emptyset, A\neq \emptyset)\right)_{A'=S|J'_m, A=S|J_n}\, .
\end{align*}
In the final expression all the cumulants $\kappa[\psi(x')_{A'},\phi(x)_{A}]$ are such that 
$|A'| < m$ and $|A|<n$.  Therefore, the induction hypothesis can be applied to estimate their $\ell_2$-norms:
\begin{align*}
& \biggl[\sum_{x\in Z^A} | \kappa[\psi(x')_{A'},\phi(x)_{A}]|^2\biggr]^{1/2} 
 \leq (\mathfrak{M}_{m',n'}\gamma^{m'})^{|S|} |S|!
|_{m'=\left|S|J'_m\right|, n'=\left|S|J_n\right|}\, .
\end{align*}
Note that $ \mathfrak{M}_{m,n}$ is non-decreasing in $m$, and $ m' \leq m-1$; hence,
\begin{align*}
& \norm{\pazocal{Q}(x', \cdot)}_{\ell_2} \leq \sum_{\pi \in \mathcal{P}(J'_m+ J_n)} \cf(|\pi| > 1) \prod_{S \in \pi} (\mathfrak{M}_{m',n'}\gamma^{m'})^{|S|} |S|!|_{m'=\left|S|J'_m\right|, n'=\left|S|J_n\right|} \\\nonumber
& \quad \leq  \sum_{\pi \in \mathcal{P}(J'_m+ J_n)}  \prod_{S \in \pi} (\mathfrak{M}_{m,n}\gamma^{m-1})^{|S|}|S|! 
%\\\nonumber & \quad 
=  (\mathfrak{M}_{m,n} \gamma^{m-1})^{m+n}\sum_{\pi \in \mathcal{P}(J'_m+ J_n)}  \prod_{S \in \pi} |S|! \nonumber \\
& \quad \leq (\mathfrak{M}_{m,n} \gamma^{m-1})^{m+n}\rme^{n+m}(n+m)! = (\mathfrak{M}_{m,n} \gamma^{m})^{m+n}(\rme/ \gamma)^{n+m}(n+m)!
\end{align*}
where in the last inequality we have used \eqref{eq:comb_est}. Collecting the two estimates together we have proven
\begin{align}
\norm{\pazocal{P}(x', \cdot)}_{\ell_2} + \norm{\pazocal{Q}(x', \cdot)}_{\ell_2} \leq  \mathfrak{M}^{n+m}_{m,n}(n+m)! [(2 \rme )^{n+m}+(\rme / \gamma)^{n+m} \gamma^{m(m+n)}]\,.
\end{align}
In order to close the induction, we need to choose $\gamma$ such that 
$$
(2 \rme )^{n+m}+(\rme / \gamma)^{n+m} \gamma^{m(m+n)} \leq \gamma^{m(m+n)}\,.
$$
Since $ m \geq 2$, it suffices to set, for instance, $\gamma = 2 \rme $.

The proof of the second item,
with $p=2$, is essentially the same: one only needs to replace the flat $\ell_2$-norm
by the above weighted $\ell_2$-norm containing the factor $|\Phi_n(y,x)|$ (which can also be understood as
integrals over the corresponding weighted counting measure over $Z$), and to apply item 2 in Theorem \ref{th_hom} instead of 
item 1 there. To reach the same combinatorial estimates, we can reduce the resulting second powers of $|S|!$ 
to sums over first powers via the bound $\prod_{S \in \pi} |S|! \le (n+m)!$ which is valid for any partition $\pi \in \mathcal{P}(J'_m+ J_n)$.
\end{proof}

\section{Discussion with an application to DNLS} 
\label{sec:discussion}

Suppose that for each $t\ge 0$ there is given $\psi_t(x,\sigma)$ which is a random field on $\Z^d\times I$ for some finite index set $I$.
Suppose also that all $\psi_t$ are identically distributed, according to an $\ell_1$-clustering measure;
such fields arise, for instance, from stochastic processes by choosing the initial data 
from a stationary measure which is $\ell_1$-clustering.  For such a system,
using $X=\psi_0(0,\sigma_0)$ in Theorem \ref{th_hom} implies
that any time-correlation function of the form
\begin{align}
 F_{t,n,\sigma_0}(x,\sigma) := \kappa[\psi_0(0,\sigma_0),\psi_t(x_1,\sigma_1),\ldots,\psi_t(x_{n-1},\sigma_{n-1})]
\end{align}
belongs to $\ell_2((\Z^d\times I)^{n-1})$ and \defem{its norm is uniformly bounded in $t$} by a constant
which depends only on the initial measure.

As an explicit example, let us come back to the discrete NLS evolution and its equilibrium
time-correlations, as discussed in the Introduction.  At the time of writing, we are not aware of a rigorous
definition of the infinite volume dynamics for an equilibrium measure of the DNLS.  However, DNLS evolution is
well-defined on any finite periodic lattice, and there is a range of hopping amplitudes and equilibrium parameters
for which the corresponding thermal Gibbs states are $\ell_1$-clustering, uniformly in the lattice size, as proven in \cite{abdesselam-2009}.
Therefore, it seems likely that there are harmonic couplings for which 
also the DNLS evolution equations on $\Z^d$ with initial data distributed according to a stationary measure
can be solved almost surely.  In addition, it should be possible to define the stationary measure so that it is
$\ell_1$-clustering and translation invariant.

For any such $\ell_1$-clustering stationary measure, we could then study the evolution of the functions $F_{t,n,\sigma_0}$ using the above results.
Referring to \cite{LM15} for details, for instance $f_t(x):=\kappa[\psi_0(0,-1),\psi_t(x_1,1)]=\E [\w{\psi_0(0,-1)}\psi_t(x_1,1)]$ would then satisfy an evolution equation
\begin{align}\label{eq:ftevoleq}
 \ci \partial_t f_t(x) = \sum_{y\in \Z^d} \alpha(x-y) f_t(y) + \lambda g_t(x)\, ,
\end{align}
where $g_t(x):=\E[\w{\psi_0(0,-1)}\psi_t(x,-1)\psi_t(x,1)\psi_t(x,1)]$.  Applying Proposition \ref{th:W_mult_prop}, 
$g_t$ can be represented in terms 
of the constant $\E[\psi_t(x,1)]$ and the functions
$F_{t,n,-1}$, with $n=2,3,4$.  Using the above estimates, it then follows
that there is a constant $C$ such that $\norm{g_t}_{\ell_2}\le C$ for all $t$.  
Taking a Fourier-transform of 
(\ref{eq:ftevoleq}) and solving it in Duhamel form implies that $\FT{f}_t=\mathcal{F}f_t$ satisfies 
\begin{align}%\label{eq:ftevoleq}
 \FT{f}_t(k) = \rme^{-\ci t \FT{\alpha}(k)} \FT{f}_0(k) -\ci \lambda \int_0^t \!\rmd s\, \rme^{-\ci (t-s) \FT{\alpha}(k)}
 \FT{g}_s(k)\, .
\end{align}
For stable harmonic interactions one needs to have $\FT{\alpha}(k)\ge 0$.  Therefore, for such systems we may conclude,
without any complicated analysis of oscillatory integrals or graph expansion of cumulants, that 
the harmonic evolution dominates the behavior of $f_t$ up to times of order $\lambda^{-1}$.  More precisely,
we find that the $\ell_2$-norm of the error is bounded by
\begin{align}\label{eq:ftbound}
\left\Vert f_t - \mathcal{F}^{-1}(\rme^{-\ci t \FT{\alpha}} \FT{f}_0)\right\Vert_{\ell_2} =
 \left\Vert \FT{f}_t - \rme^{-\ci t \FT{\alpha}} \FT{f}_0\right\Vert_{L^2(\T^d)} \le C t \lambda \, ,
\end{align}
for \defem{all} $t\ge 0$.

The above example perhaps does not appear very significant: after all, it is simply stating that the nonlinearity acts as a perturbation in $\ell_2$-norm with its ``natural'' strength, having an effect of order $\lambda t$ to the time-evolution.
Let us however stress that without the present a priori bounds there seems to be no other alternative 
to prove this than to resort to the heavy machinery of time-dependent perturbation theory with Feynman graph classification of oscillatory integrals and careful applications of momentum hierarchies, see \cite{NLS09} for a detailed example for DNLS. 

Another important property hidden in the bound (\ref{eq:ftbound}) is the fact that $\FT{f}_t(k)$ is a \defem{function}, and not a distribution.  This would \defem{not} be true in general if instead of cumulants we would have used moments to define $f_t$;
if $\E[\psi_0(0)]\ne 0$ and the initial state is translation invariant, 
already at $t=0$ the Fourier transform of $\E[\psi_0(0,-1) \psi_t(x,1)]$ is a distribution proportional to the Dirac delta
$\delta(k)$.  The fact that the cumulants produce functions, which are uniformly bounded in $\ell_2$, allows not only taking 
Fourier transforms but also simplifying the study of nonlinear terms in the hierarchies as products of distributions are
notoriously difficult to control rigorously.

On the mathematical side, it would be of interest to study more carefully the above combinatorial bounds.  
We do not claim that the above constants, or their dependence on the orders $n$ and $m$, should be optimal,
and there could be room for significant improvement there, possibly of importance in 
problems requiring the full infinite order cumulant hierarchy.  Also, it is not clear what are the 
optimal powers and weights for the summability of the correlations
as the clustering power $p$ of the field $\phi$ is varied. These questions could prove to be hard to resolve in 
the greatest generality, but we remain optimistic that already the present bounds suffice to control 
the time-evolution of cumulants in some of the above mentioned open transport problems.

\appendix

 \section{Cumulants and Wick polynomials}\label{appendix}

We consider a collection $y_j$, $j\in J$ where $J$ is some fixed nonempty index 
set, of real or complex random variables
on some probability space $(\Omega,\mathcal{M},\mu)$.
Then for any sequence of indices, $I=(i_1,i_2,\ldots,i_n)\in J^n$,
we use following shorthand notations
to label monomials of the above random variables:
\begin{align}\label{eq:defIpower}
 y^I = y_{i_1} y_{i_2}\cdots y_{i_n} = \prod_{k=1}^n y_{i_k} \, ,\qquad y^\emptyset := 1\quad \text{if}\quad I=\emptyset. \,% is the empty sequence
\end{align}

We introduce the collection $\mathcal{I}$ which consists of those finite subsets $A\subset 
\N\times J$ with the
property that if $(n,j), (n',j')\in A$ and $(n,j)\ne (n',j')$
then $n\ne n'$.  The  empty sequence is identified with 
$\emptyset\in \mathcal{I}$.
For nonempty sets, the natural number in the first component serves as a 
distinct label for each member in $A$.
As already explained in an earlier footnote, 
we consider sequences of indices and not sets of indices in order to avoid a more cumbersome notation due to the possible relabellings of the sets. 
For any $I\in \mathcal{I}$ we denote the corresponding moment by $ \E[y^I]$, and
the related cumulant by
\begin{align}
 & \kappa[y_I] = \kappa_\mu[y_I] = \kappa[y_{i_1}, y_{i_2}, \cdots{}, y_{i_n}]\, .
\end{align}
The corresponding Wick polynomial is denoted by
\begin{align}
\w{y^I} = \w{y^I}_\mu =\w{y_{i_1} y_{i_2} \cdots y_{i_n}}\, .
\end{align}
Both $\kappa[y^I]$ and $ \w{y^I}$ can be defined recursively if $I\in \mathcal{I}$ is such that $\E[|y^E|]<\infty$ for all 
$E\subset I$ (see \cite{LM15}).  Explicitly, it suffices to require that
\begin{align}\label{eq:defTJ}
  \w{y^{I}} = y^I - \sum_{E\subsetneq I} \E[y^{I\setminus E}]\,
  \w{y^E}\, ,
\end{align}
and, choosing some $x\in I$,
\begin{align}\label{eq:iterdefkappa}
\kappa[y_I] = 
 \E[y^{I}] - \sum_{E:x\in E\subsetneq I} \E[y^{I\setminus E}] \kappa[y_E]\, .
\end{align}
Let us also recall that both cumulants and Wick polynomials are multilinear and permutation invariant. 

If the random variables $y_j$, $j=1,2,\ldots,n$, have joint exponential moments, 
then moments, 
cumulants and Wick polynomials can also be easily generated by differentiation of their 
respective generating functions which are
\begin{align}\label{eq:defallgenf}
\Mgen(\lambda) := \E[\rme^{\lambda\cdot x}]\, , \quad
\Cgen(\lambda) := \ln \Mgen(\lambda) \qand
\Wgen(\lambda; y) := \frac{\rme^{\lambda\cdot y}}{\E[\rme^{\lambda\cdot x}]} =
\rme^{\lambda\cdot y-\Cgen(\lambda)}
\, .
\end{align}
By evaluation of the $I$-th derivative at zero, we have
\begin{align}
\E[y^I] = \partial^I_\lambda\Mgen(0)\, , \quad
\kappa[y_I]= \partial^I_\lambda\Cgen(0) \qand
\w{y^I} = \partial^I_\lambda\Wgen(0;y)
\, ,
\end{align}
where ``$\partial^I_\lambda$'' is a shorthand notation for 
$\partial_{\lambda_{i_1}}\partial_{\lambda_{i_2}}\cdots 
\partial_{\lambda_{i_n}}$.

It is remarkable that expectations of products of Wick polynomials can be expanded in terms of cumulants,
merely cancelling some terms from the standard moments-to-cumulants expansion.
The following result, proven as Proposition 3.8 in \cite{LM15},
details the result using the above notations:
\begin{proposition}\label{th:W_mult_prop}
Assume that the measure $\mu$ has all moments of order $N$, i.e.,  suppose that 
$\E[|y^I|]<\infty$ for all $I\in \mathcal{I}$ with 
$|I|\le N$.  Suppose $L\ge 1$ is given and
consider a collection of $L+1$ index sequences
$J',J_\ell\in\mathcal{I}$, $\ell=1,\ldots,L$, such that $|J'|+\sum_\ell 
|J_\ell|\le N$.
Then for $I:= \sum_{\ell=1}^L J_\ell + J'$ 
(with the implicit identification of $J_\ell$ and $J'$ with the set of its 
labels in $I$) we have 
\begin{equation}
\E \bigg[ \prod_{\ell=1}^L \w{y^{J_\ell} }  y^{J'}\bigg] = \sum_{\pi \in 
\pazocal{P}(I)}\prod_{A \in \pi}
\!\left(\kappa[y_A] \cf(A \not\subset J_\ell\  \forall \ell)\right) \, .
\label{wick_prod_multi}
\end{equation}
\end{proposition}
In words, the constraint 
determined by the characteristic functions on the right hand side of \eqref{wick_prod_multi}
amounts to removing from the standard cumulant expansion all terms which have any clusters internal to one of the sets $J_\ell$. 
For instance, thanks to Proposition \ref{th:W_mult_prop}, if we consider the expectation of the product of two second order Wick polynomials,
we get 
$\E[\w{y_1y_2} \w{y_3y_4}]= \kappa( y_1 , y_3   )\kappa( y_2 , y_4)+ \kappa( y_1 , y_4   )\kappa( y_2 , y_3)+\kappa(y_1,y_2,y_3,y_4)$.

Proposition \ref{th:W_mult_prop} turns out to be a powerful technical tool, used several times in the proofs of Theorems \ref{th_hom} and \ref{th_weigheted}.

% \newcommand{\utildir}[1]{../../texstuff/#1}
% % .bst -file
% \bibliographystyle{\utildir{abunst_titles}}
% % .bib-files
% \bibliography{\utildir{myabbr},\utildir{mrabbrev},\utildir{allrefs}}
% \end{document}

\end{document}